\documentclass[11pt,twoside]{amsart}
\usepackage{latexsym,amssymb,amsmath}

\numberwithin{equation}{section}
\def\demo{\noindent{\it Proof. }}
\newtheorem{theorem}{Theorem}[section]
\newtheorem{lemma}[theorem]{Lemma}
\newtheorem{proposition}[theorem]{Proposition}
\newtheorem{corollary}[theorem]{Corollary}
\newtheorem{conjecture}[theorem]{Conjecture}

\theoremstyle{definition}
\newtheorem{definition}[theorem]{Definition} 
 
\newtheorem{remark}[theorem]{Remark}
\newtheorem{example}[theorem]{Example}

\begin{document}

%%%%%%%%%%%%%%%%%%%%%%%%%%%%%%%%%%%%%%%%%%%%%%%%%%%%%%%%%%%%%%%%%%%%%

\title[Footprint and minimum
distance functions]{Footprint and minimum distance functions}  

\thanks{The first and third author were supported by SNI. The second author
was supported by CONACYT and Red Tem\'atica Matem\'aticas y Desarrollo}

\author{Luis N\'u\~nez-Betancourt}
\address{
Centro de Investigaci\'on en Matem\'aticas, 
Guanajuato, Gto., M\'exico.
}
\email{luisnub@cimat.mx}

\author{Yuriko Pitones}
\address{
Departamento de
Matem\'aticas\\
Centro de Investigaci\'on y de Estudios
Avanzados del
IPN\\
Apartado Postal
14--740 \\
07000 Mexico City, D.F.
}
\email{ypitones@math.cinvestav.mx}

\author{Rafael H. Villarreal}
\address{
Departamento de
Matem\'aticas\\
Centro de Investigaci\'on y de Estudios
Avanzados del
IPN\\
Apartado Postal
14--740 \\
07000 Mexico City, D.F.
}
\email{rvillarreal@cinvestav.mx}
%\urladdr{http://www.math.cinvestav.mx/$\sim$vila/}

\begin{abstract} 
Let $S$ be a polynomial ring over a field $K$,  
with a monomial order $\prec$, and let $I$ be an unmixed graded ideal
of $S$. In this paper we study two functions
associated to $I$: the minimum distance
function $\delta_I$ and the footprint function ${\rm fp}_I$. It is
shown that $\delta_I$ is positive and that ${\rm fp}_I$ is positive if
the initial ideal of $I$ is
unmixed. Then we show
that if $I$ is radical and its associated primes are generated
by linear forms, then $\delta_I$ is strictly decreasing until it
reaches the asymptotic value $1$. If $I$ is the edge ideal of a
Cohen--Macaulay bipartite graph, we show that $\delta_I(d)=1$ for 
$d$ greater than or equal to the regularity of $S/I$. 
For a graded ideal of dimension $\geq 1$, whose initial ideal is a complete
intersection, we give an exact sharp 
lower bound for the corresponding minimum distance function.
\end{abstract}

\maketitle 

\section{Introduction}\label{intro-section}
Let $S=K[t_1,\ldots,t_s]=\oplus_{d=0}^{\infty} S_d$ be a polynomial ring over
a field $K$ with the standard grading and let $I\neq(0)$ be a graded ideal
of $S$. The {\it degree\/} or {\it multiplicity\/} of $S/I$ is
denoted by $\deg(S/I)$.

Given an integer $d\geq 1$, let $\mathcal{F}_d$ be the set of 
all zero-divisors of $S/I$ not in $I$ of degree $d\geq 1$:
$$
\mathcal{F}_d:=\{\, f\in S_d\, \vert\, f\notin I,\, (I\colon f)\neq
I\},
$$
where $(I\colon f):=\{h\in S\vert\, hf\in I\}$ is the {\it quotient
\/} ideal or {\it colon\/} ideal of $I$ with
respect to $f$. The {\it minimum distance function\/} of $I$ is the function  
$\delta_I\colon \mathbb{N}_+\rightarrow \mathbb{Z}$ given by 
$$
\delta_I(d):=\left\{\begin{array}{ll}\deg(S/I)-\max\{\deg(S/(I,f))\vert\,
f\in\mathcal{F}_d\}&\mbox{if }\mathcal{F}_d\neq\emptyset,\\
\deg(S/I)&\mbox{if\ }\mathcal{F}_d=\emptyset.
\end{array}\right.
$$

Fix a graded monomial order $\prec$ on $S$. The initial ideal of $I$
is denoted by ${\rm in}_\prec(I)$. Let $\Delta_\prec(I)$  be the 
{\it footprint\/} or \emph{Gr\"obner \'escalier} 
of $S/I$ consisting of all the {\it standard monomials\/} of $S/I$,
that is, all the monomials of $S$ not in the ideal ${\rm in}_\prec(I)$. 

Let $\mathcal{M}_{\prec, d}$ be the set of 
all zero-divisors of $S/{\rm in}_\prec(I)$ of degree $d\geq 1$ that
are in $\Delta_\prec(I)$:
$$
\mathcal{M}_{\prec, d}:=\{t^a\, \vert\,
t^a\in\Delta_\prec(I)\cap S_d,\ ({\rm in}_\prec(I)\colon t^a)\neq {\rm
in}_\prec(I)\}.
$$ 

The {\it footprint
function\/} of $I$, 
denoted ${\rm fp}_I$, is the function ${\rm fp}_I\colon
\mathbb{N}_+\rightarrow \mathbb{Z}$ given
by 
$$
{\rm fp}_I(d):=\left\{\begin{array}{ll}\deg(S/I)-\max\{\deg(S/({\rm
in}_\prec(I),t^a))\,\vert\,
t^a\in\mathcal{M}_{\prec, d}\}&\mbox{if }\mathcal{M}_{\prec, d}\neq\emptyset,\\
\deg(S/I)&\mbox{if }\mathcal{M}_{\prec, d}=\emptyset.
\end{array}\right.
$$

In this paper we study $\delta_I$ and ${\rm fp}_I$ from a theoretical
point of view. The functions $\delta_I$ and ${\rm fp}_I$ were
introduced in \cite{hilbert-min-dis} and \cite{min-dis-ci},
respectively. 
The interest in these functions is essentially due to the following
two facts: the minimum distance function is related to the minimum
distance in coding theory \cite[Theorem~4.7]{hilbert-min-dis} and the
footprint function is much easier to compute. There are significant
cases in which either the footprint function is a lower bound for the
minimum distance function \cite[Lemma~3.10(a)]{min-dis-ci} or the two
functions coincide \cite[Corollary~4.4]{min-dis-ci}. 

The footprint lower bound was used
in the works of Geil \cite{geil} and Carvalho \cite{carvalho} to study
affine Reed-Muller-type codes. Long before these two papers appeared the
footprint was used by Geil in connection with all kinds
of codes (including one-point algebraic geometric codes); see
\cite{geil-2008,geil-hoholdt,geil-pellikaan} and the references therein.

The contents of this
paper are as follows. In Section~\ref{prelim-section} we 
present some of the results and terminology that will be needed
in the paper. 

Our first result shows that $\delta_I$ is positive if $I$ is unmixed,
and that ${\rm fp}_I$ is also
positive if
${\rm in}_\prec(I)$ is unmixed (Theorem~\ref{md-unmixed}). This
improves the correlated non-negativity of the functions $\delta_I$ 
and ${\rm fp}_I$ that was shown in \cite[Lemma~3.10]{min-dis-ci}. We show
that if $I$ is a radical unmixed ideal whose associated primes are generated
by linear forms, then $\delta_I$ is strictly decreasing until it
reaches the asymptotic value $1$ (Theorem~\ref{md-decreasing}). This 
gives a wide generalization of
\cite[Theorem~4.5(vi)]{hilbert-min-dis}. Then 
we conjecture that $\delta_I(d)=1$ for $d\geq{\rm reg}(S/I)$, where ${\rm reg}(S/I)$ is
the regularity of $S/I$ (Conjecture~\ref{md=1-regularity-conjecture}).
We show this conjecture when $I$ is the
edge ideal of a Cohen--Macaulay bipartite graph without isolated
vertices (Proposition~\ref{md=1-regularity-cmbg}). 

If $I$ is a complete intersection monomial ideal of dimension
$\geq 1$, we present an explicit formula for ${\rm fp}_{I}(d)$ 
(Theorem~\ref{footprint-ci-monomial}). In this case  ${\rm
fp}_{I}(d)=\delta_I(d)$
(Proposition~\ref{unmixed-monomial-geil-carvalho}).   
For a graded ideal of dimension $\geq 1$, whose initial ideal is a complete
intersection, we give an exact sharp 
lower bound for the corresponding minimum distance function 
(Theorem~\ref{footprint-ci-general}). As a particular case we recover
\cite[Theorem~3.14]{min-dis-ci}; 
as is seen in \cite{min-dis-ci} this result has interesting
applications to coding theory and to packing and covering in
combinatorics.

\smallskip

For all unexplained
terminology and additional information,  we refer to 
\cite{BHer,CLO} (for the theory of Gr\"obner bases, 
commutative algebra, and Hilbert functions).

\section{Preliminaries}\label{prelim-section}
All results of this
section are well-known. To avoid repetitions we continue to employ
the notations and 
definitions used in Section~\ref{intro-section}.

Let $S=K[t_1,\ldots,t_s]=\oplus_{d=0}^{\infty} S_d$ be a polynomial ring over
a field $K$ with the standard grading and let $I\neq(0)$ be a graded ideal
of $S$ of dimension $k$. By the dimension of $I$ we mean the Krull
dimension of $S/I$.   
The {\it Hilbert function} of $S/I$, denoted $H_I$, is given by: 
$$
H_I(d):=\dim_K(S_d/I_d),\ \ \ d=0,1,2,\ldots,
$$
where $I_d=I\cap S_d$. By a theorem of Hilbert \cite[Theorem~4.1.3]{BHer}  
there is a unique polynomial
$h_I(t)\in\mathbb{Q}[t]$ of 
degree $k-1$ such that $h_I(d)=H_I(d)$ for  $d\gg 0$. 
By convention the degree of the zero polynomial is $-1$.  

The {\it degree\/} or {\it multiplicity\/} of $S/I$, denoted
$\deg(S/I)$, is the 
positive integer 
$$
\deg(S/I):=\left\{\begin{array}{ll}(k-1)!\,\displaystyle
\lim_{d\rightarrow\infty}{H_I(d)}/{d^{k-1}} 
&\mbox{if }k\geq 1,\\
\dim_K(S/I) &\mbox{if\ }k=0.
\end{array}\right.
$$ 

\begin{definition}
If $I$ is a graded ideal of $S$, the {\it Hilbert series\/} of $S/I$, 
denoted $F_I(x)$, is given by 
$$F_I(x)=\sum_{d=0}^\infty H_I(d)x^d, \mbox{ where }x\mbox{ is a
variable}. $$ 
\end{definition}
\begin{theorem}{\rm(Hilbert--Serre \cite[p.~58]{Sta1})}\label{hilbert-serre}
Let $I\subset S$ be a graded ideal of dimension $k$. 
Then there is a unique polynomial
$h(x)\in\mathbb{Z}[x]$ such that
$$
F_I(x)=\frac{h(x)}{(1-x)^k}\ \mbox{ and }\ h(1)>0. 
$$
\end{theorem}

\begin{remark}\label{mult-vs-h-vector-1}\rm The leading coefficient 
of the Hilbert polynomial $h_I(x)$ is equal to
$h(1)/(k-1)!$. Thus $h(1)$ is equal to $\deg(S/I)$. 
\end{remark}

\begin{definition} Let $I\subset S$ be a graded ideal. 
The $a$-invariant of $S/I$, denoted $a(S/I)$, is the degree of
$F_I(x)$ as a rational 
function, that is, $a(S/I)=\deg(h(x))-k$.
\end{definition}

\begin{definition}\rm Let $I\subset S$ be a graded ideal and let 
${\mathbb F}_\star$ be the minimal graded free resolution of $S/I$ as an 
$S$-module: 
\[
{\mathbb F}_\star:\ \ \ 0\rightarrow 
\bigoplus_{j}S(-j)^{b_{gj}}
\stackrel{}{\rightarrow} \cdots 
\rightarrow\bigoplus_{j}
S(-j)^{b_{1j}}\stackrel{}{\rightarrow} S
\rightarrow S/I \rightarrow 0.
\]
The {\it Castelnuovo--Mumford regularity\/}\index{Castelnuovo--Mumford}
of $S/I$ ({\it regularity} of $S/I$ for short) \index{regularity} 
is defined as 
$${\rm reg}(S/I)=\max\{j-i\vert\,
b_{ij}\neq 0\}.$$ 
\end{definition}

An excellent reference for the regularity of graded ideals is the book of Eisenbud
\cite{eisenbud-syzygies}. The $a$-invariant, the regularity, and the 
depth of $S/I$ are closely related.

\begin{theorem}{\rm\cite[Corollary~B.4.1]{Vas1}}\label{reg-cm} $a(S/I)\leq{\rm 
reg}(S/I)-{\rm depth}(S/I)$, with equality if $S/I$ is Cohen--Macaulay.
\end{theorem}

\begin{definition}\label{definition:index-of-regularity}\rm 
The \emph{regularity index} of the Hilbert function of $S/I$, or simply
the \emph{regularity index} of $S/I$, denoted 
${\rm ri}(S/I)$, is the least integer $n\geq 0$ such that
$H_I(d)=h_I(d)$ for $d\geq n$.  
\end{definition}

If $I$ is a graded
Cohen-Macaulay ideal of $S$ of dimension $1$, then ${\rm reg}(S/I)$,
the regularity of $S/I$ 
is equal to ${\rm ri}(S/I)$, the regularity index of $S/I$. This follows from
Theorem~\ref{reg-cm}. 

\begin{definition}\label{ci-def} An ideal $I\subset S$ is called a {\it complete intersection\/} if 
there exist $g_1,\ldots,g_{r}$ in $S $ such that $I=(g_1,\ldots,g_{r})$, 
where $r={\rm ht}(I)$ is the height of $I$. 
\end{definition}

\begin{remark} 
(a) A graded ideal $I$ is a
complete intersection if and only if $I$ is generated by a homogeneous regular 
sequence with ${\rm ht}(I)$ elements (see
\cite[Chapter~3]{Kaplansky}). (b) A monomial
ideal $I$ is a complete intersection if and only if $I$ is minimally
generated by a regular sequence of monomials with ${\rm ht}(I)$
elements.
\end{remark}

\begin{lemma}{\cite[Corollary~3.3]{Sta1}}\label{hilbertseries-ci} 
If $I\subset S$ is an ideal generated by 
homogeneous polynomials $f_1,\ldots,f_r$, with $r={\rm ht}(I)$ and 
$\delta_i=\deg(f_i)$, then the Hilbert series 
of $S/I$ is given by 
\begin{equation*}
F_I(x)=
\frac{\prod_{i=1}^{r}\left(1-x^{\delta_i}\right)}{(1-x)^s}.
\end{equation*}
\end{lemma}

\begin{lemma}{\rm(}\cite[Example~1.5.1]{Migliore},
\cite[Lemma~3.5]{Chardin}{\rm)}\label{hilbertseries-ci-main} If $I\subset S$ is a
complete intersection ideal generated by 
homogeneous polynomials $f_1,\ldots,f_r$, with $r={\rm ht}(I)$ and 
$\delta_i=\deg(f_i)$, then the degree and regularity of $S/I$ are
given by $\deg(S/I)=\delta_1\cdots \delta_r$ and 
${\rm reg}(S/I)=\sum_{i=1}^r(\delta_i-1)$.
\end{lemma}

\begin{proof} 
The formula for the degree follows from
Remark~\ref{mult-vs-h-vector-1} and Lemma~\ref{hilbertseries-ci}. 
As $S/I$ is Cohen--Macaulay, the formula for the regularity follows from 
Lemma~\ref{hilbertseries-ci} and Theorem~\ref{reg-cm}.
\end{proof}

If $f$ is a non-zero 
polynomial in $S$ and $\prec$ is a monomial order on $S$, we denote
the {\it leading monomial\/} of $f$ by ${\rm in}_\prec(f)$. 
For $a=(a_1,\ldots,a_s)\in\mathbb{N}^s$, we set
$t^a=t_1^{a_1}\cdots t_s^{a_s}$. Let $I\subset S$ be an ideal.  
A monomial $t^a$ is called a 
{\it standard monomial\/} of $S/I$, with respect 
to $\prec$, if $t^a$ is not the leading monomial of any polynomial in
$I$, that is, $t^a$ is not in the ideal ${\rm in}_\prec(I)$. 
A polynomial $f$ is called {\it standard\/} if
$f\neq 0$ and $f$ is a
$K$-linear combination of standard monomials.
The set of standard monomials, denoted $\Delta_\prec(I)$, is called the {\it
footprint\/} of $S/I$ or \emph{Gr\"obner \'escalier} of $I$. 
A subset $\mathcal{G}=\{g_1,\ldots, g_r\}$ of $I$ is called a 
{\it Gr\"obner basis\/} of $I$ if $${\rm
in}_\prec(I)=({\rm in}_\prec(g_1),\ldots,{\rm in}_\prec(g_r)).$$ 

An element $f\in S$ is called a {\it zero-divisor\/} of $S/I$---as an
$S$-module---if there is
$\overline{0}\neq \overline{a}\in S/I$ such that
$f\overline{a}=\overline{0}$, and $f$ is called {\it regular\/} on
$S/I$ otherwise. Notice that $f$ is a zero-divisor if
and only if $(I\colon f)\neq I$.

\begin{lemma}{\rm\cite[Lemma~2.8]{min-dis-ci}}\label{regular-elt-in} 
Let $\prec$ be a monomial order, let $I\subset S$ be an ideal, and let
$f$ be a polynomial of $S$ of positive degree. If ${\rm in}_\prec(f)$
is regular on $S/{\rm in}_\prec(I)$, then $f$ is regular on $S/I$. 
\end{lemma}

An \emph{associated prime} of $I$ is a prime
ideal $\mathfrak{p}$ of $S$ of the form $\mathfrak{p}=(I\colon f)$
for some $f$ in $S$. An ideal $I\subset S$ is called {\it unmixed\/} 
if all its associated primes have the same height and $I$ is called
{\it radical\/} if $I$ is equal to its radical. 

\begin{definition}\rm 
If ${\rm fp}_I(d)=\delta_I(d)$ for $d\geq 1$, we say that $I$ is a
{\it Geil--Carvalho ideal\/}.
\end{definition}

\begin{proposition}{\rm\cite[Proposition~3.11]{min-dis-ci}}
\label{unmixed-monomial-geil-carvalho} If $I$ is an unmixed monomial
ideal and $\prec$ is 
any monomial order, then $\delta_I(d)={\rm fp}_I(d)$ for $d\geq 1$,
that is, $I$ is a Geil--Carvalho ideal.
\end{proposition}

\begin{proposition}\label{md-unmixed-propo}
Let $I\subset S$ be an unmixed graded ideal, let $\prec$ be a monomial
order on $S$, and let $d\geq 1$ be an
integer. The following hold. 
\begin{itemize}
\item[\rm(a)] {\rm\cite[Lemma~3.10(a)]{min-dis-ci}} $\delta_I(d)\geq {\rm fp}_I(d)$.
\item[\rm(b)] {\rm\cite[Theorem~4.5(iv)]{hilbert-min-dis}} If $t_i$
is a zero-divisor of 
$S/I$ for all $i$, then ${\rm fp}_I(d)\geq 0$.
\end{itemize}
\end{proposition}

The lower bound of Proposition~\ref{md-unmixed-propo}(b) is sharp. In 
Example~\ref{footprint=0-unmixed-ti-zero-div} we show an unmixed
graded ideal $I$ of dimension $1$ such that $t_i$ is a zero-divisor 
for all $i$ and ${\rm fp}_I(d)=0$ for $d=1$.

\begin{proposition}{\rm(Additivity of the degree
\cite[Proposition~2.5]{prim-dec-critical})}\label{additivity-of-the-degree}
If $I$ is an ideal of $S$ and 
$I=\mathfrak{q}_1\cap\cdots\cap\mathfrak{q}_m$ 
is an irredundant primary
decomposition, then
$$
\deg(S/I)=\sum_{{\rm ht}(\mathfrak{q}_i)={\rm
ht}(I)}\hspace{-3mm}\deg(S/\mathfrak{q}_i).$$
\end{proposition}

The additivity is one of the most useful and well-known facts about the
degree.

\section{Minimum and footprint functions}\label{min-dis-section}
In this section we study the footprint and minimum distance functions
of unmixed graded ideals over an arbitrary field.

\begin{lemma}\label{degree-initial-footprint}
Let $I\subset S$ be an unmixed graded ideal and let $\prec$ be 
a monomial order. If $f\in S$ is homogeneous and $(I\colon f)\neq I$, then
\begin{itemize}
\item[\rm(i)] {\rm\cite[Lemma~4.1]{hilbert-min-dis}}
$\deg(S/(I,f))\leq\deg(S/({\rm
in}_\prec(I),{\rm in}_\prec(f)))\leq\deg(S/I)$,
\item[\rm(ii)] $\deg(S/I)=\deg(S/(I\colon f))+\deg(S/(I, f))$ if
$f\notin I$, and 
\item[\rm(iii)] $\deg(S/(I,f))<\deg(S/I)$ if $f\notin I$. 
\end{itemize}
\end{lemma}

\begin{proof} (ii): Using that $I$ is unmixed, it is not hard to see that 
$S/I$, $S/(I\colon f)$, and $S/(I,f)$ have the same Krull dimension.
There is an exact sequence
\begin{eqnarray*}
0\longrightarrow
S/(I\colon f)[-d]\stackrel{f}{\longrightarrow}
S/I\longrightarrow
S/(I,f)\longrightarrow 0.
\end{eqnarray*}
Hence, by the additivity of Hilbert functions \cite[Lemma~5.1.1]{monalg-rev}, we get 
\begin{equation}\label{dec1-15-1}
H_I(i)=H_{(I\colon f)}(i-d)+H_{(I,f)}(i)\ \mbox{
for }i\geq 0.
\end{equation}

If $\dim S/I=0$, then using Eq.~(\ref{dec1-15-1}) one has
$$
\sum_{i\geq 0}H_I(i)=\sum_{i\geq 0}H_{(I\colon f)}(i)+
\sum_{i\geq 0}H_{(I,f)}(i).
$$

Therefore, using the definition of degree, the
required equality follows. If $k=\dim S/I-1$ and $k\geq 1$, by the
Hilbert theorem \cite[Theorem~4.1.3]{BHer}, 
$H_{I}$, $H_{(I,f)}$, and $H_{(I\colon f)}$ are polynomial
functions of degree $k$. Then dividing 
Eq.~(\ref{dec1-15-1}) by $i^{k}$ and taking limits as $i$ goes to
infinity, the required equality follows.

(iii): This part follows at once from part (ii).
\end{proof}

The next alternative formula for $\delta_I$ 
is valid for unmixed graded ideals. This expression for $\delta_I$
will be used to show some of our results.

\begin{corollary}{\rm\cite[Theorem~4.4]{hilbert-min-dis}}\label{wolmer-obs}
Let $I\subset S$ be an unmixed graded ideal. 
If $\mathfrak{m}=(t_1,\ldots,t_s)$ and $d\geq 1$ is an integer such
that $\mathfrak{m}^d\not\subset I$, then
\begin{equation*}
\delta_I(d)=\min\{\deg(S/(I\colon f))\,\vert\, f\in S_d\setminus
I\}.
\end{equation*}
\end{corollary}

%\begin{proof}
\demo  
If $\mathcal{F}_d=\emptyset$, then 
$\delta_I(d)=\deg(S/I)$, and for any $f\in S_d\setminus I$ one has that
$(I\colon f)$ is equal to $I$. Thus equality holds.
Assume that $\mathcal{F}_d\neq\emptyset$. Take $f\in S_d\setminus I$.
If $(I\colon f)=I$, then $\deg(S/(I\colon f))$ is equal to
$\deg(S/I)$. On the other hand if $(I\colon f)\neq I$, that is,
$f\in\mathcal{F}_d$, then by Lemma~\ref{degree-initial-footprint}(ii)
one has the equality:
\begin{equation*}%\label{dec8-15} 
\deg(S/(I\colon f))=\deg(S/I)- \deg(S/(I, f)).
\end{equation*}
Notice that in this case $\deg(S/(I\colon f))\leq\deg(S/I)$. 
Therefore 
\begin{eqnarray*}
\delta_I(d)&=&\deg(S/I)-\max\{\deg(S/(I,f))\vert\,
f\in\mathcal{F}_d\}\\
&=&\min\{\deg(S/(I\colon f))\,\vert\, f\in \mathcal{F}_d\}\\ 
&=&
\min\{\deg(S/(I\colon f))\,\vert\, f\in S_d\setminus
I\}. \ \ \ \Box
\end{eqnarray*}
%\end{proof}

\begin{definition}\label{vasconcelos-function}
Let $I\subset S$ be a non-zero proper graded ideal. The {\it Vasconcelos function\/}
of $I$ is the function
$\vartheta_I\colon\mathbb{N}_+\rightarrow \mathbb{N}_+$ given by 
$$
\vartheta_I(d)=\left\{\begin{array}{ll}\min\{\deg(S/(I\colon f))\,\vert\, f\in S_d\setminus
I\}&\mbox{if }\mathfrak{m}^d\not\subset I,\\
\deg(S/I)&\mbox{if }\mathfrak{m}^d\subset I.\end{array}\right.
$$
\end{definition}

Very little is known about the Vasconcelos function when $I$ is not an unmixed
graded ideal. Next we show that in certain cases the footprint
function can be expressed in terms of the degree of colon ideals. 

\begin{corollary}\label{jul1-16} Let $I$ be a graded ideal and let $\prec$
be a monomial order. If ${\rm in}_\prec(I)$ is an unmixed ideal and 
$\mathcal{M}_{\prec, d}\neq\emptyset$, then 
$$
{\rm fp}_I(d)
=\min\{\deg(S/({\rm
in}_\prec(I)\colon t^a))\,\vert\,
t^a\in S_d\setminus{\rm
in}_\prec(I)\}.
$$
\end{corollary} 
\demo Take $t^a\in\mathcal{M}_{\prec,d}$. 
By Lemma~\ref{degree-initial-footprint}(ii)
one has the equality:
\begin{equation*}%\label{dec8-15} 
\deg(S/({\rm in}_\prec(I)\colon t^a))=\deg(S/{\rm in}_\prec(I))-
\deg(S/({\rm in}_\prec(I), t^a)).
\end{equation*}
In this case $\deg(S/({\rm in}_\prec(I)\colon t^a))\leq\deg(S/{\rm
in}_\prec(I))$. 
Therefore, noticing that $\deg(S/{\rm in}_\prec(I))$ is equal to 
$\deg(S/I)$, we get
\begin{eqnarray*}
{\rm fp}_I(d)
&=&\deg(S/I)-\max\{\deg(S/({\rm in}_\prec(I),t^a))\vert\,
t^a\in\mathcal{M}_{\prec,d}\}\\
&=&\min\{\deg(S/({\rm in}_\prec(I)\colon t^a))\,\vert\, t^a\in
\mathcal{M}_{\prec, d}\}\\
&=&\min\{\deg(S/({\rm
in}_\prec(I)\colon t^a))\,\vert\,
t^a\in S_d\setminus {\rm
in}_\prec(I)\}.\ \ \ \Box
\end{eqnarray*}

One can apply the corollary to graded lattice ideals of dimension $1$. 

\begin{proposition} Let $I\subset S$ be a graded lattice ideal of 
dimension $1$ and let $\prec$ be a graded monomial order with
$t_1\succ\cdots\succ t_s$. The following hold.
\begin{itemize}
\item[\rm(a)] If ${\rm in}_\prec(I)$ is not prime, then ${\rm
in}_\prec(I)$ is unmixed and $\mathcal{M}_{\prec,d}\neq\emptyset$ for
$d\geq 1$. 
\item[\rm(b)] If ${\rm in}_\prec(I)$ is prime, then 
$I=(t_1-t_s,\ldots,t_{s-1}-t_s)$ and 
$\mathcal{M}_{\prec,d}=\emptyset$ for $d\geq 1$.
\end{itemize}
\end{proposition}
\begin{proof} The reduced Gr\"obner basis of $I$ consists of binomials
of the form $t^{a_+}-t^{a_-}$ (see
\cite[Proposition~8.2.7]{monalg-rev}). It follows that $t_s$ is a
regular element on both $S/I$ and $S/{\rm in}_\prec(I)$. Hence 
$I$ and ${\rm in}_\prec(I)$ are Cohen--Macaulay ideals. 
In particular these ideals are unmixed. 

(a): Assume that ${\rm in}_\prec(I)$ is not prime. 
Then there is an associated prime $\mathfrak{p}$ of $S/{\rm
in}_\prec(I)$ such that ${\rm in}_\prec(I)\subsetneq\mathfrak{p}$.
Pick a variable $t_i$ in $\mathfrak{p}\setminus{\rm in}_\prec(I)$.
Then $t_it_s^{d-1}$ is in $\mathfrak{p}$ and is not in ${\rm
in}_\prec(I)$ for $d\geq 1$. Thus $t_it_s^{d-1}$ is in
$\mathcal{M}_{\prec,d}$ for $d\geq 1$. 

(b): Assume that ${\rm in}_\prec(I)$ is
prime. This part follows by noticing that ${\rm
in}_\prec(I)$, being a face ideal generated by variables, 
is equal to $(t_1,\ldots,t_{s-1})$.
\end{proof} 

The next result is a broad generalization of \cite[Lemma~3.10]{min-dis-ci}.

\begin{theorem}\label{md-unmixed}
Let $I\subset S$ be an unmixed graded ideal, let $\prec$ be a monomial
order on $S$, and let $d\geq 1$ be an
integer. The following hold. 
\begin{itemize}
\item[\rm(a)] $\delta_I(d)\geq 1$.
\item[\rm(b)] ${\rm fp}_I(d)\geq 1$ if ${\rm in}_\prec(I)$ is unmixed. 
\item[\rm(c)] If $\dim(S/I)\geq 1$ and $\mathcal{F}_d\neq\emptyset$ for $d\geq 1$, then  
$\delta_I(d)\geq \delta_I(d+1)\geq 1$ 
for $d\geq 1$. 
\end{itemize}
\end{theorem}

\begin{proof} 
(a): If $\mathcal{F}_d=\emptyset$, then $\delta_I(d)=\deg(S/I)\geq 1$,
and if $\mathcal{F}_d\neq\emptyset$, then using
Lemma~\ref{degree-initial-footprint}(iii) it follows
that $\delta_I(d)\geq 1$.  

(b): If $\mathcal{M}_{\prec,d}=\emptyset$, then 
${\rm fp}_I(d)=\deg(S/I)\geq 1$. Next assume that
$\mathcal{M}_{\prec,d}\neq\emptyset$. 
As ${\rm in}_\prec(I)$ is unmixed, by
Corollary~\ref{jul1-16}, ${\rm fp}_I(d)\geq 1$.

(c): By part (a), one has $\delta_I(d)\geq 1$. The 
set $\mathcal{F}_d$ is not empty for $d\geq 1$. Thus, by
Corollary~\ref{wolmer-obs}, $\delta_I(d)=\deg(S/(I\colon f))$ 
for some $f\in\mathcal{F}_d$. 
As $I$ is unmixed and $\dim(S/I)\geq 1$, $\mathfrak{m}$ is not an
associated prime of $S/I$. Thus, since $(I\colon f)$ is a graded
ideal, one has $(I\colon f)\subsetneq\mathfrak{m}$. 
Pick a linear form $h\in S_1$ such that $hf\notin I$. 
As $f$ is a zero-divisor of $S/I$, so is $hf$. The ideals
$(I\colon f)$ and $(I\colon hf)$ have height equal to ${\rm ht}(I)$. 
Therefore taking Hilbert
functions in the exact sequence 
$$
0\longrightarrow (I\colon hf)/(I\colon f)\longrightarrow S/(I\colon f)\longrightarrow
S/(I\colon hf)\longrightarrow 0
$$
it follows that $\deg(S/(I\colon f))\geq \deg(S/(I\colon hf))$.
Therefore, applying Corollary~\ref{wolmer-obs}, we get the 
inequality $\delta_I(d)\geq \delta_I(d+1)$. 
\end{proof}

\begin{lemma}\label{jul11-15-colon} 
Let $I\subset S$ be a radical unmixed graded ideal and let
$\mathfrak{p}_1,\ldots\mathfrak{p}_m$ be its associated primes. 
If $f\in\mathcal{F}_d$ for some $d\geq 1$, then
$$
\deg(S/(I\colon f))=\sum_{f\notin\mathfrak{p}_i}\deg(S/\mathfrak{p}_i).
$$
\end{lemma}

\begin{proof} Since
$I$ is a radical ideal, we get that $I=\cap_{i=1}^m\mathfrak{p}_i$.
From the equalities
\[
(I\colon f)=\cap_{i=1}^m(\mathfrak{p}_i\colon
f)=\cap_{f\notin\mathfrak{p}_i}\mathfrak{p}_i,
\]
and using the additivity of the degree 
(see Proposition~\ref{additivity-of-the-degree}), 
the required equality follows.
\end{proof}

We come to the main result of this section---about the asymptotic
behavior of the minimum distance function--which gives a wide generalization of
\cite[Theorem~4.5(vi)]{hilbert-min-dis}. 

\begin{theorem}\label{md-decreasing}
Let $I\subset S$ be an unmixed radical graded ideal. If all the associated primes 
of $I$ are generated by linear forms, then there is an integer $r_0\geq 1$ such that 
$$
\delta_I(1)>\cdots>\delta_I(r_0)=\delta_I(d)=1\
\mbox{ for }\ d\geq r_0.
$$
\end{theorem}

\begin{proof} Let $\mathfrak{p}_1,\ldots,\mathfrak{p}_m$ be the
associated primes of $I$. As $\mathfrak{p}_i$ is generated by linear
forms, then $\deg(S/\mathfrak{p}_i)=1$ for all $i$. Indeed if
$\mathfrak{p}_i=\mathfrak{m}$, then $\deg(S/\mathfrak{p}_i)$ is
$\dim_K(S/\mathfrak{p}_i)=1$, and if
$\mathfrak{p}_i\subsetneq\mathfrak{m}$, then the initial ideal of 
$\mathfrak{p}_i$, with respect to the GRevLex order $\prec$, is generated
by a subset of $t_1,\ldots,t_s$ and $\deg(S/\mathfrak{p}_i)$ is equal
to $\deg(S/{\rm
in}_\prec(\mathfrak{p}_i))=1$. The last equality follows noticing 
that $S/{\rm in}_\prec(\mathfrak{p}_i)$ is a polynomial ring.

If $I$ is prime, then $I=\mathfrak{p}_i$ for some $i$ and
$\mathcal{F}_d=\emptyset$ for $d\geq 1$. Thus $\delta_I(d)=
\deg(S/\mathfrak{p}_i)=1$ for $d\geq 1$, and we can take $r_0=1$. We may now
assume that $I$ has at least two associated primes, that is, 
$m\geq 2$. As $I\subsetneq\mathfrak{p}_1$, there is a form $h$ of
degree $1$ in 
$\mathfrak{p}_1\setminus I$. Hence, as $I$ is a radical ideal, we get
that $h^d$ is in $\mathfrak{p}_1\setminus I$. Thus
$\mathcal{F}_d\neq\emptyset$ for $d\geq 1$. Therefore, by
Theorem~\ref{md-unmixed}(c), one has that $\delta_I(d)\geq \delta_I(d+1)\geq 1$ 
for $d\geq 1$. Hence, assuming that $\delta_I(d)>1$, it suffices to
show that $\delta_I(d)>\delta_I(d+1)$. By Corollary~\ref{wolmer-obs},
there is $f\in \mathcal{F}_d$
such that $\delta_I(d)=\deg(S/(I\colon f))$. 
Then, by Lemma~\ref{jul11-15-colon}, one has
$$
\delta_I(d)=\deg(S/(I\colon f))=
\sum_{f\notin\mathfrak{p}_i}\deg(S/\mathfrak{p}_i)\geq 2.
$$

Hence there are $\mathfrak{p}_k\neq\mathfrak{p}_j$ such that $f$ is
not in $\mathfrak{p}_k\cup \mathfrak{p}_j$. Pick a linear form $h$
in $\mathfrak{p}_k\setminus\mathfrak{p}_j$. 
Then $h f\notin I$
because $h f\notin\mathfrak{p}_j$, and $h f$ is a
zero-divisor of $S/I$ because $(I\colon f)\neq I$. Noticing that 
$f\notin\mathfrak{p}_k$ and $h f\in\mathfrak{p}_k$, one obtains the
strict inclusion
$$
\{\mathfrak{p}_i\vert\, hf\notin\mathfrak{p}_i\}\subsetneq 
\{\mathfrak{p}_i\vert\, f\notin\mathfrak{p}_i\}.
$$
Therefore, by Lemma~\ref{jul11-15-colon}, we get
$$
\deg(S/(I\colon
f))=\sum_{f\notin\mathfrak{p}_i}\deg(S/\mathfrak{p}_i)>
\sum_{h f\notin\mathfrak{p}_i}\deg(S/\mathfrak{p}_i)=\deg(
S/(I\colon h f)).
$$
Hence, by Corollary~\ref{wolmer-obs}, we get $\delta_I(d)>\delta_I(d+1)$.
\end{proof}

\section{Asymptotic behavior of the minimum
distance}\label{conjecture-md-reg}

Let $I\subset S$ be an unmixed radical graded ideal whose associated
primes 
are generated by linear forms. 
According to Theorem~\ref{md-decreasing}, there is an integer
$r_0\geq 1$  such that
$$
\delta_I(1)>\cdots>\delta_I(r_0)=\delta_I(d)=1
\ \mbox{ for }\ d\geq r_0.
$$

\begin{definition}
The integer $r_0$ is called the {\it regularity index\/} of
$\delta_I$. 
\end{definition}
\indent If $I$ is the graded vanishing ideal of a set of points in a
projective space over a finite field, then $r_0\leq {\rm reg}(S/I)$
\cite{GRT,algcodes}, but we do not know whether this holds in
general. The regularity
of $S/I$ can be computed using {\it Macaulay\/}$2$ \cite{mac2}, 
but $r_0$ is very difficult to compute.      

\begin{conjecture}\label{md=1-regularity-conjecture}
Let $I\subset S$ be an unmixed radical graded ideal. If all the associated primes 
of $I$ are generated by linear forms,  then 
$\delta_I(d)=1$ for $d\geq{\rm reg}(S/I)$, that is, $r_0\leq {\rm reg}(S/I)$.
\end{conjecture}

In this section we give some support for this conjecture.
In what follows we focus in the case that $I$ is an unmixed ideal generated by square-free
monomial ideals of degree $2$.  

\begin{definition}\cite{Vi2}\rm\ Let $G$ be a graph with vertex 
set $V(G)=\{t_1,\ldots,t_s\}$ and edge set $E(G)$.
The {\it edge
ideal\/} of $G$, denoted by $I(G)$,
is the ideal of $S$ generated by all  
monomials $x_e=\prod_{t_i\in e}t_i$ such 
that $e\in E(G)$. 
\end{definition}

Let $G$ be a graph. A subset $F$ of $V(G)$ is called {\it
stable\/} if $e\not\subset F$ for any  
$e\in E(G)$, and a subset $C$ of $V(G)$ is a {\it vertex
cover\/} if and only if $V(G)\setminus C$ is a stable vertex set. A 
{\it minimal vertex cover\/} is a vertex cover which
is minimal with respect to inclusion. A graph is called {\it
unmixed\/} if all its minimal vertex covers have the same cardinality.

Conjecture~\ref{md=1-regularity-conjecture} is open even in the case that $I$ is the edge ideal of
an unmixed bipartite graph. Below we prove the conjecture for edge
ideals of Cohen-Macaulay graphs.

\begin{definition}\rm Let $A$ be a set of vertices of a graph $G$.  
The {\it induced subgraph\/} on $A$, denoted by 
$G[A]$, is the maximal
subgraph of $G$ with vertex set $A$. A
graph of the form $G[A]$ for some $A\subset
V(G)$ is called an {\it induced subgraph} of
$G$. 
\end{definition}

Notice that
$G[A]$ may have isolated vertices, i.e., vertices that do
not belong to any edge of $G[A]$. If $G$ is a
discrete graph, i.e., 
all the vertices of $G$ are
isolated, 
we set $I(G)=0$. 

\begin{definition}\rm An {\it induced matching\/} in a graph $G$ is a set of
pairwise disjoint edges $f_1,\ldots,f_r$ such that the only edges of 
$G$ contained in $\cup_{i=1}^rf_i$ are $f_1,\ldots,f_r$. The
{\it induced matching number\/},
denoted by ${\rm im}(G)$, is the number of edges in the largest
induced matching. 
\end{definition}

\begin{proposition}{\rm\cite[Lemma~2.2]{katzman1}}\label{lower-bound-reg-g} 
If $G$ is a graph, then ${\rm
reg}(R/I(G))\geq {\rm im}(G)$.
\end{proposition}

Next we prove Conjecture~\ref{md=1-regularity-conjecture} for edge
ideals of Cohen--Macaulay bipartite graphs. A graph $G$ is called 
{\it Cohen--Macaulay\/} if $S/I(G)$ is Cohen--Macaulay. 

\begin{proposition}\label{md=1-regularity-cmbg} If $I=I(G)$ is the
edge ideal of a Cohen--Macaulay bipartite graph without isolated
vertices, then  
$\delta_I(d)=1$ for $d\geq{\rm reg}(S/I)$.
\end{proposition}

\begin{proof} By \cite[Theorem~1.1]{kummini}, ${\rm reg}(S/I)={\rm im}(G)$. Thus, by
Theorem~\ref{md-decreasing}, it suffices to show that $\delta_I(d)=1$
for some $d\leq {\rm im}(G)$. 
According to \cite[Theorem~3.4]{herzog-hibi-crit},
there is a bipartition $V_1=\{x_1,\ldots,x_g\}$,
$V_2=\{y_1,\ldots,y_g\}$ of $G$ such that: 

\smallskip

\noindent {\rm (a)} $e_i=\{x_i,y_i\}\in E(G)$ for all $i$, 

\smallskip

\noindent{\rm (b)} if $\{x_i,y_j\}\in E(G)$, then $i\leq j$, and 

\smallskip

\noindent {\rm (c)} if $\{x_i,y_j\}$, $\{x_j,y_k\}$ are in $E(G)$ and $i<j<k$, 
then $\{x_i,y_k\}\in E(G)$.

\smallskip

Next we construct a sequence $x_{i_1},\ldots,x_{i_d}$ such that 
$e_{i_1},\ldots,e_{i_d}$ form an induced matching and $V_2$ is a
pairwise disjoint union   
\begin{equation}\label{mar1-16}
V_2=N_G(x_{i_1})\cup\cdots\cup N_G(x_{i_d}),
\end{equation}
where $N_G(x_{i_j})\cap N_G(x_{i_k})=\emptyset$ for $j\neq k$ and 
$N_G(x_{i_j})$ is
the neighbor set of $x_{i_j}$, that is, $N_G(x_{i_j})$ is the set of vertices of $G$
adjacent to $x_{i_j}$. We set $i_1=1$. If $N_G(x_{i_1})\subsetneq V_2$, pick
$y_{i_2}$ in $V_2\setminus N_G(x_{i_1})$. By condition (b),
$e_{i_1},e_{i_2}$ is an induced matching and $N_G(x_{i_1})\cap
N_G(x_{i_2})=\emptyset$. If $N_G(x_{i_1})\cup N_G(x_{i_2})\subsetneq V_2$, pick
$y_{i_3}$ in $V_2\setminus(N_G(x_{i_1})\cup N_G(x_{i_2}))$. By condition (b),
$e_{i_1},e_{i_2},e_{i_3}$ form an induced matching and $N_G(x_{i_j})\cap
N_G(x_{i_k})=\emptyset$ for $j\neq k$. Thus one can continue this process until we get
a sequence $x_{i_1},\ldots,x_{i_d}$ such that 
$V_2$ is the disjoint union of the $N_G(x_{i_j})$'s and the
$e_{i_j}$'s form an induced matching. 

Let
$\mathfrak{p}_1,\ldots,\mathfrak{p}_m$ be the associated primes 
of $I$. There are minimal vertex covers $C_1,\ldots,C_m$ of $G$ such
that $\mathfrak{p}_i$ is generated by $C_i$ for $i=1,\ldots,m$ (see
\cite[p.~279]{Vi2}). We may
assume that $C_m=V_2$. 
Setting $x^a=x_{i_1}\cdots x_{i_d}$, by Corollary~\ref{wolmer-obs}, it
suffices to show that
$x^a$ is in $\cap_{i=1}^{m-1}\mathfrak{p}_i\setminus\mathfrak{p}_m$ and 
that $\deg(S/(I\colon x^a))=1$, where $S=K[V(G)]$. If $i\neq m$, there is 
$y_\ell\notin C_i$. From Eq.~(\ref{mar1-16}), there is $x_{i_j}$ such
that $y_\ell\in N_G(x_{i_j})$ for some $i_j$. Hence, as $C_i$ covers the
edge $\{x_{i_j},y_\ell\}$, one has that $x_{i_j}$ is in
$\mathfrak{p}_i$. Thus $x^a$ is in $\cap_{i=1}^{m-1}\mathfrak{p}_i$
and $x^a$ is not in $\mathfrak{p}_m$ because
$\mathfrak{p}_m=(y_1,\ldots,y_g)$. Therefore
$$
(I\colon x^a)=(\mathfrak{p}_1\cap\cdots\cap\mathfrak{p}_m\colon x^a)=
(\mathfrak{p}_1\colon x^a)\cap\cdots\cap(\mathfrak{p}_m\colon
x^a)=\mathfrak{p}_m.
$$
Hence $\deg(S/(I\colon x^a))=1$, as required.
\end{proof}

\section{Complete intersections}\label{ci-section}

Let $S=K[t_1,\ldots,t_s]=\oplus_{d=0}^{\infty} S_d$ be a polynomial ring over a field $K$ with
the standard grading and let $\prec$ be a graded monomial order. 

\begin{proposition}\label{jan2-16} 
Let $I\subset S$ be a graded ideal and let $\prec$ be a
monomial order. Suppose that ${\rm in}_\prec(I)$ is a complete intersection of
height $r$ generated by 
$t^{\alpha_1},\ldots,t^{\alpha_r}$ with $d_i=\deg(t^{\alpha_i})$ and 
$d_i\geq 1$ for all $i$. The following hold.
\begin{itemize}
\item[\rm(a)] {\rm \cite[Example~1.5.1]{Migliore}} 
$I$ is a complete intersection and $\dim(S/I)=s-r$.
\item[\rm(b)] $\deg(S/I)=d_1\cdots d_r$ and 
${\rm reg}\, S/I=\sum_{i=1}^r(d_i-1)$. 
\item[\rm(c)] $1\leq {\rm fp}_I(d)\leq \delta_I(d)$ for $d\geq 1$.
\end{itemize}
\end{proposition}
\begin{proof} (a): The rings $S/I$ and $S/{\rm init}_\prec(I)$ have the same
dimension. Thus $\dim(S/I)=s-r$. As $\prec$ is a graded order, there are 
$f_1,\ldots,f_r$ homogeneous polynomials in $I$ with 
${\rm in}_\prec(f_i)=t^{\alpha_i}$ for $i\geq 1$. Since 
$$
{\rm in}_\prec(I)=({\rm in}_\prec(f_1),\ldots,{\rm in}_\prec(f_r)),
$$
the polynomials $f_1,\ldots,f_r$ form a Gr\"obner basis of $I$, and in
particular they generated $I$. Hence $I$ is a graded ideal of height
$r$ generated by $r$ polynomials, that is, $I$ is a complete intersection.

(b): This follows at once from part (a) and 
Lemma~\ref{hilbertseries-ci-main}. 

(c): By part (a), $I$ is a complete
intersection. In particular $I$ is a Cohen--Macaulay unmixed ideal. Hence this part follows from
Proposition~\ref{md-unmixed-propo} and Theorem~\ref{md-unmixed}. 
\end{proof}

\begin{lemma}\label{reducing-exponents-lemma} 
Let $I\subset S$ be a complete intersection ideal 
minimally generated by $t^{\alpha_1},\ldots,t^{\alpha_r}$ and let
$t^a=t_1^{a_1}\cdots t_s^{a_s}$ be a zero-divisor of $S/I$ not in $I$. 
The following hold.
\begin{itemize}
\item[\rm(a)] $t^{\alpha_i}$ and $t^{\alpha_j}$ have no common
variable for $i\neq j$.
\item[\rm(b)] If $t_j^{a_j}$ is regular on $S/I$ and
$t^c=t^a/t_j^{a_j}$, then
$(I\colon t^a)=(I\colon t^c)$.
\item[\rm(c)] If $t_j$ is a zero-divisor of $S/I$, then there is 
a unique $\alpha_i=(\alpha_{i,1},\ldots,\alpha_{i,s})$ such that
$\alpha_{i,j}>0$, that is, $t_j$ occurs in exactly one
$t^{\alpha_i}$. If $a_j>\alpha_{i,j}$ and $t^c=t^a/t_j$, then 
$(I\colon t^a)=(I\colon t^c)$.
\item[\rm(d)] For each $i$ there is 
$t^{\beta_i}$ 
dividing $t^{\alpha_i}$ such that 
$\deg(t^{\beta_i})<\deg(t^{\alpha_i})$ and $(I\colon t^a)=(I\colon t^\beta)$,
where $t^\beta=t^{\beta_1}\cdots t^{\beta_r}$. 
\end{itemize} 
\end{lemma}

\begin{proof} (a): This follows readily from the Krull principal
ideal theorem \cite[Theorem~2.3.16]{monalg-rev}. 

(b): The inclusion ``$\supset$'' is clear. To show the reverse
inclusion take $t^\delta$ in $(I\colon t^a)$, that is, $t^\delta
t^a=t^\delta t_j^{a_j}t^c$ is in $I$. Hence $t^\delta t^c$ is in $I$
because $t_j^{a_j}$ is regular on $S/I$. Thus $t^\delta$ is in
$(I\colon t^c)$.

(c): If $t_j$ is a zero-divisor of $S/I$, then $t_j$ is in some
associated prime of $S/I$. Hence, by part (a), $t_j$ must occur in a
unique $t^{\alpha_i}$ for some $i$. Thus one has $\alpha_{i,j}>0$. We
claim that $((t^{\alpha_k})\colon t^a)=((t^{\alpha_k})\colon t^c)$ for
all $k$. If $k\neq i$, by part (a), $t_j$ is regular on
$S/(t^{\alpha_k})$. Thus, as in the proof of part (b), we get the
asserted equality. Next we assume that $k=i$. The inclusion ``$\supset$'' is clear. To show the reverse
inclusion take $t^\delta$ in $((t^{\alpha_i})\colon t^a)$, that is, $t^\delta
t^a=t^\gamma t^{\alpha_i}$ for some $t^\gamma$. Since
$a_j>\alpha_{i,j}>0$, $t_j$ must divide $t^\gamma$. Then we can write
$t^\delta t^c=t^\omega t^{\alpha_i}$, where $t^\omega=t^\gamma/t_j$.
Thus $t^\delta$ is in $((t^{\alpha_i})\colon t^c)$. This completes the
proof of the claim. Therefore one has
\begin{eqnarray*}
(I\colon t^a)&=&((t^{\alpha_1})\colon
t^a)+\cdots+((t^{\alpha_r})\colon t^a)\\ 
&=&((t^{\alpha_1})\colon t^c)+\cdots+((t^{\alpha_r})\colon
t^c)=(I\colon t^c).
\end{eqnarray*}

(d): Using part (a) and successively applying parts (b) and (c) to
$t^a$, we get a monomial $t^\beta$ that divides $t^a$ such that the
following conditions are satisfied: 
(i) all variables that occur in $t^\beta$ are zero-divisors of $S/I$,
(ii) if $t^\beta=t_1^{\gamma_1}\cdots t_s^{\gamma_s}$ and
$\gamma_j>0$, then $\alpha_{i,j}\geq \gamma_j$, where $t^{\alpha_i}$ is
the unique monomial, among $t^{\alpha_1},\ldots,t^{\alpha_r}$,
containing $t_j$, and 
(iii) $(I\colon t^a)=(I\colon t^\beta)$. We let $t^{\beta_i}$ be the
product of all $t_j^{\gamma_j}$ such that $t_j$ occurs in
$t^{\alpha_i}$. Clearly $t^{\beta_i}$ divides $t^{\alpha_i}$, and
$\deg(t^{\alpha_i})>\deg(t^{\beta_i})$ because $t^a$ 
is not in $I$ by hypothesis.
\end{proof}

The next result gives some additional support to
Conjecture~\ref{md=1-regularity-conjecture}. 

\begin{proposition}\label{footprint-monomial-ci-ineq} 
Let $I\subset S$ be a complete intersection monomial ideal of dimension $\geq 1$
minimally generated by $t^{\alpha_1},\ldots,t^{\alpha_r}$. If
$d_i=\deg(t^{\alpha_i})$ for $i=1,\ldots,r$. The following hold. 
\begin{itemize}
\item[\rm(a)] ${\rm reg}(S/I)=\sum_{i=1}^r(d_i-1)$,
\item[\rm(b)] $\delta_I(d)=1$ if $d\geq {\rm reg}(S/I)$,
\item[\rm(c)] $\delta_I(d)\leq\left(d_{k+1}-\ell\right)d_{k+2}\cdots d_r$
if $d<{\rm reg}(S/I)$, 
where $0\leq k\leq r-1$ and $\ell$ are integers such that 
$d=\sum_{i=1}^{k}\left(d_i-1\right)+\ell$ and $1\leq \ell \leq
d_{k+1}-1$. 
\end{itemize}
\end{proposition}

\begin{proof} (a): This follows at once from 
Lemma~\ref{hilbertseries-ci-main}. 

(b): By Lemma~\ref{reducing-exponents-lemma}(a) the monomials $t^{\alpha_i}$ and
$t^{\alpha_j}$ have no common variables for $i\neq j$. For 
each $i$ pick $t_{j_i}$ in $t^{\alpha_i}$. If $I$ is prime, then
$I=(t_{j_1},\ldots,t_{j_r})$, ${\rm reg}(S/I)=0$, 
$\mathcal{F}_d=\emptyset$ and $\delta_I(d)=1$ for $d\geq 1$.
Thus we may assume that $I$ is not prime. We claim that
$\mathcal{F}_d\neq\emptyset$ for $d\geq 1$. As $I$ is not prime, there is 
$m$ such that $t_{j_m}$ a zero-divisor of $S/I$ not in $I$.  
If a variable $t_n$ is not in $t^{\alpha_i}$ for any $i$, then $t_n$
is a regular element on $S/I$, and $\mathcal{F}_d\neq\emptyset$ 
because $t_{j_m}t_n^{d-1}$ is in $\mathcal{F}_d$. If
any variable $t_n$ is in $t^{\alpha_i}$ for some $i$, then any
monomial of degree $d$ is a zero-divisor of $S/I$ because any 
variable $t_n$ belongs to at least one associated prime of $S/I$.
As $\dim(S/I)\geq 1$, one has $\mathfrak{m}^d\not\subset I$. Pick a
monomial $t^a$ of degree $d$ not in $I$. Then
$\mathcal{F}_d\neq\emptyset$ because $t^a$ is in $\mathcal{F}_d$. 
This completes the proof of the claim. 
We set $t^{c_i}=t^{\alpha_i}/t_{j_i}$ for
$i=1,\ldots,r$ and
$t^c=t^{c_1}\cdots t^{c_r}$. Then it is seen that 
$(I\colon t^c)=(t_{j_1},\ldots,t_{j_r})$ and $\deg S/(I\colon
t^c)=1$. Notice that $t^c$ is a zero-divisor of $S/I$, 
$t^c\notin I$ and $\deg(t^c)={\rm reg}(S/I)$. 
Hence, by Corollary~\ref{wolmer-obs}, we get that
$\delta_I(d)=1$ for $d={\rm reg}(S/I)$. Thus, by
Theorem~\ref{md-unmixed}(c), 
we get $\delta_I(d)=1$ for $d\geq {\rm reg}(S/I)$.

(c): There is a monomial $t^a$ of degree $\ell$ 
that divides $t^{\alpha_{k+1}}$ because $\ell$ is a positive integer 
less than or equal to $d_{k+1}-1$. Setting $t^c=t^{c_1}\cdots
t^{c_k}t^{a}$ and $t^\gamma=t^{\alpha_{k+1}}/t^a$, one has 
$$(I\colon
t^c)=(t_{j_1},\ldots,t_{j_k},t^{\gamma},t^{\alpha_{k+2}},\ldots,t^{\alpha_r}).
$$ 
Hence, by Lemma~\ref{hilbertseries-ci-main}, we get 
$\deg S/(I\colon t^c)=(d_{k+1}-\ell)d_{k+2}\cdots d_r$ 
because $(I\colon t^c)$ is a complete intersection. 
Since $\deg(t^c)=d=\sum_{i=1}^{k}\left(d_i-1\right)+\ell$, $t^c$ is
not in $I$, and $t^c$ is a zero-divisor of $S/I$, by
Corollary~\ref{wolmer-obs} we get 
that $\deg S/(I\colon t^c)\geq \delta_I(d)$, as required. 
\end{proof}

\begin{proposition}{\rm \cite[Proposition~5.7]{hilbert-min-dis}}\label{aug-28-15}
Let $1\leq e_1\leq\cdots\leq e_m$ and $0\leq b_i\leq e_i-1$
for $i=1,\ldots,m$ be integers. If $b_0\geq 1$, then 
\begin{equation}\label{aug-27-15-2}
\prod_{i=1}^m(e_i-b_i)\geq
\left(\sum_{i=1}^{k+1}(e_i-b_i)-(k-1)-b_0-\sum_{i=k+2}^m b_i\right)e_{k+2}\cdots e_m
\end{equation}
for $k=0,\ldots,m-1$, where $e_{k+2}\cdots e_m=1$ and
$\sum_{i=k+2}^mb_i=0$ if $k=m-1$.
\end{proposition}

We come to the main result of this section.

\begin{theorem}\label{footprint-ci-monomial} 
Let $I\subset S$ be a complete intersection monomial ideal of dimension $\geq 1$
minimally generated by $t^{\alpha_1},\ldots,t^{\alpha_r}$ and let
$d\geq 1$ be an integer. If
$d_i=\deg(t^{\alpha_i})$ for $i=1,\ldots,r$ and $d_1\leq\cdots\leq d_r$, then
$$
\delta_I(d)={\rm fp}_I(d)=\left\{\hspace{-1mm}
\begin{array}{ll}\left(d_{k+1}-\ell\right)d_{k+2}\cdots d_r&\mbox{ if }
d<\sum\limits_{i=1}^{r}\left(d_i-1\right),\\
\qquad \qquad 1&\mbox{ if } d\geq
\sum\limits_{i=1}^{r}\left(d_i-1\right),
\end{array}
\right.
$$
where $0\leq k\leq r-1$ and $\ell$ are integers such that 
$d=\sum_{i=1}^{k}\left(d_i-1\right)+\ell$ and $1\leq \ell \leq
d_{k+1}-1$. 
\end{theorem}

\begin{proof} The ideal $I$ is unmixed because $I$ is Cohen--Macaulay. Hence, by
Proposition~\ref{unmixed-monomial-geil-carvalho}, $I$ is
Geil--Carvalho, that is, $\delta_I(d)={\rm fp}_I(d)$ for $d\geq 1$.
Therefore, by Proposition~\ref{footprint-monomial-ci-ineq}, it
suffices to show that 
$$
{\rm fp}_I(d)\geq(d_{k+1}-\ell)d_{k+2}\cdots d_r\ \mbox{ for } 
d<{\rm reg}(S/I).
$$

 Let $t^a$ be a monomial of degree $d$ such that $t^a\notin I$ and
$(I\colon t^a)\neq I$. By Lemma~\ref{reducing-exponents-lemma}(d), for
each $i$ there is a monomial $t^{\beta_i}$ dividing $t^{\alpha_i}$ such that 
$\deg(t^{\beta_i})<\deg(t^{\alpha_i})$ and $(I\colon t^a)=(I\colon t^\beta)$,
where $t^\beta=t^{\beta_1}\cdots t^{\beta_r}$. One can write 
$$ 
t^{\alpha_i}=t_1^{\alpha_{i,1}}\cdots t_s^{\alpha_{i,s}}\ \mbox{ and
}\  t^{\beta_i}=t_1^{\beta_{i,1}}\cdots t_s^{\beta_{i,s}}
$$
for $i=1,\ldots,r$. According to
Lemma~\ref{reducing-exponents-lemma}(a) the monomials 
$t^{\alpha_i}$ and $t^{\alpha_j}$ have no
common variables for $i\neq j$. As $(I\colon t^\beta)$ is a monomial
ideal, it follows that  
$$
(I\colon t^a)=(I\colon t^\beta)=(\{t_1^{\alpha_{i,1}-\beta_{i,1}}\cdots
t_s^{\alpha_{i,s}-\beta_{i,s}}\}_{i=1}^r).
$$

Hence, setting $g_i=t_1^{\alpha_{i,1}-\beta_{i,1}}\cdots
t_s^{\alpha_{i,s}-\beta_{i,s}}$ for $i=1,\ldots,r$ and observing that 
$g_i$ and $g_j$ have no common variables for $i\neq j$, we get that
$g_1,\ldots,g_r$ form a regular sequence, that is, $(I\colon t^a)$ is
again a complete intersection. Thus, by
Lemma~\ref{hilbertseries-ci-main}, we obtain 
$$
\deg(S/(I\colon
t^a))=\prod_{i=1}^r\left[\sum_{j=1}^s(\alpha_{i,j}-\beta_{i,j})\right]=
\prod_{i=1}^r\left[\deg(t^{\alpha_i})-\deg(t^{\beta_i})\right].
$$
Therefore, setting $b_i=\deg(t^{\beta_i})$ for $i=1,\ldots,r$, we get
$$
\deg(S/(I\colon
t^a))=\prod_{i=1}^r(d_i-b_i).
$$
Thus, by Corollary~\ref{wolmer-obs}, it suffices to show the inequality 
$$
\deg(S/(I\colon
t^a))=\prod_{i=1}^r(d_i-b_i)\geq(d_{k+1}-\ell)d_{k+2}\cdots d_r.
$$
Noticing that 
$d=\deg(t^a)=\sum_{i=1}^{k}\left(d_i-1\right)+\ell\geq\deg(t^\beta)=\sum_{i=1}^rb_i$,
one has
$$
\left(d_{k+1}+\sum_{i=1}^k(d_i-1)-\sum_{i=1}^rb_i\right)d_{k+2}\cdots
d_r\geq(d_{k+1}-\ell)d_{k+2}\cdots d_r.
$$
Hence, we need only show the inequality
$$
\prod_{i=1}^r(d_i-b_i)\geq 
\left(\sum_{i=1}^{k+1}(d_i-b_i)-k-\sum_{i=k+2}^rb_i\right)d_{k+2}\cdots
d_r,
$$
which follows making $b_0=1$ and $m=r$ in Proposition~\ref{aug-28-15}. 
\end{proof}

\begin{theorem}\label{footprint-ci-general}
Let $I\subset S$ be a graded ideal of dimension $\geq 1$ and let $\prec$ be a
monomial order. If ${\rm in}_\prec(I)$ is a complete intersection of
height $r$ generated by 
$t^{\alpha_1},\ldots,t^{\alpha_r}$ with $d_i=\deg(t^{\alpha_i})$ and 
$1\leq d_i\leq d_{i+1}$ for $i\geq 1$, then $\delta_I(d)\geq {\rm
fp}_I(d)\geq 1$ and the footprint function is given by  
$$
{\rm fp}_I(d)=\left\{\begin{array}{ll}(d_{k+1}-\ell)d_{k+2}\cdots d_r
&\mbox{if }\ 
1\leq d\leq \sum\limits_{i=1}^{r}\left(d_i-1\right)-1,\\ 
1 &\mbox{if\ }\ d\geq \sum\limits_{i=1}^{r}\left(d_i-1\right),
\end{array}\right.
$$
where $0\leq k\leq r-1$ and $\ell$ are integers such that 
$d=\sum_{i=1}^{k}\left(d_i-1\right)+\ell$ and $1\leq \ell \leq
d_{k+1}-1$. 
\end{theorem}

\begin{proof} By Proposition~\ref{jan2-16} one has $\delta_I(d)\geq {\rm
fp}_I(d)\geq 1$. Since ${\rm fp}_I(d)$ is equal to 
${\rm fp}_{{\rm in}_\prec(I)}(d)$
for $d\geq 1$, the formula for ${\rm fp}_I(d)$ follows directly from
Theorem~\ref{footprint-ci-monomial}.
 \end{proof}

It is an open question whether in Theorem~\ref{footprint-ci-general} 
one has the equality $\delta_I(d)={\rm fp}_I(d)$ for $d\geq 1$. 
If we make $r=s-1$ in Theorem~\ref{footprint-ci-general}, we recover
\cite[Theorem~3.14]{min-dis-ci}. The reader is referred to \cite{min-dis-ci} for some 
interesting applications of this result to
algebraic coding theory. As is seen in
\cite[Corollary~4.5]{min-dis-ci} this result can also be
used to extend a result of Alon and F\"uredi
\cite[Theorem~1]{alon-furedi} about coverings of the cube 
$\{0,1\}^n$ by affine hyperplanes.

\section{Computing the minimum distance function}

Let $I\subset S$ be a graded ideal and let $\prec$ be a monomial
order. The minimum distance function of $I$
can be expressed as follows. 

\begin{theorem}\label{computing-the-minimum-distance} 
If $\Delta_\prec(I)\cap
S_d=\{t^{a_1},\ldots,t^{a_n}\}$ is the set of all standard monomials
of $S/I$ of degree $d\geq 1$ and 
$\mathcal{F}_{\prec,d}=\left.\left\{f=\sum_i{\lambda_i}t^{a_i}\, \right|\,
f\neq 0,\ \lambda_i\in
K,\, (I\colon f)\neq I\right\}$, then 
$$
\delta_I(d)=
\deg(S/I)-\max\{\deg(S/(I,f))\vert\, f\in\mathcal{F}_{\prec, d}\}.
$$
\end{theorem}

\demo 
Let $f$ be any element of $\mathcal{F}_d$. Pick a Gr\"obner basis
$g_1,\ldots,g_r$ of $I$. Then, by the division algorithm
\cite[Theorem~3, p.~63]{CLO}, we can write 
$f=\sum_{i=1}^ra_ig_i+h$, 
where $h$ is a homogeneous standard polynomial of $S/I$ of degree $d$. 
Since $(I\colon f)=(I\colon h)$, we get that $h$ is in
$\mathcal{F}_{\prec,d}$. Hence, as $(I,f)=(I,h)$, we get the
equalities:
\begin{eqnarray*}
\delta_I(d)&=&
\deg(S/I)-\max\{\deg(S/(I,f))\vert\, f\in\mathcal{F}_d\}\\
&=&
\deg(S/I)-\max\{\deg(S/(I,f))\vert\, f\in\mathcal{F}_{\prec, d}\}.
\ \ \ \Box
\end{eqnarray*}

Notice that $\mathcal{F}_d\neq\emptyset$ if and only if
$\mathcal{F}_{\prec, d}\neq\emptyset$. 
If $K=\mathbb{F}_q$ is a finite field, then the number of standard
polynomials of degree $d$ is $q^n-1$, where $n$ is the
number of standard monomials of degree $d$.  
Hence, we can compute $\delta_I(d)$ for small values of
$d$, $n$, and $q$. To compute ${\rm fp}_I(d)$ is much easier 
even if the field is infinite because $\mathcal{M}_{\prec,d}$ has at
most $n$ elements. 

\begin{example}\label{dec4-15}
Let $K$  be the field $\mathbb{F}_2$ and let $I$ be the ideal of 
$S=\mathbb{F}_2[t_1,t_2,t_3]$ generated by the binomials 
$t_1t_2^2-t_1^2t_2,\, t_1t_3^2-t_1^2t_3,\, t_2^2t_3-t_2t_3^2$. 
If $S$ has the GRevLex order $\prec$, then 
using Theorem~\ref{computing-the-minimum-distance} and the procedure below
for {\em Macaulay\/}$2$ \cite{mac2} we get 
\begin{eqnarray*}
\hspace{-11mm}&&\left.
\begin{array}{c|c|c|c|c}
d & 1 & 2 & 3 & \cdots\\
   \hline
{\deg}(S/I) & 7 & 7 & 7 &\cdots
 \\ 
   \hline
 H_I(d)    \    & 3 & 6  & 7& 
 \cdots\\ 
   \hline
 \delta_I(d) &4& 2& 1& \cdots\\ 
 \hline
 {\rm fp}_I(d) &4& 1& 1& \cdots\\ 
\end{array}
\right.
\end{eqnarray*}

\begin{verbatim}
q=2
S=ZZ/q[t1,t2,t3]
I=ideal(t1*t2^q-t1^q*t2,t1*t3^q-t1^q*t3,t2^q*t3-t2*t3^q)
M=coker gens gb I, degree M, regularity M
h=(d)->degree M - max apply(apply(apply(apply(
toList (set(0..q-1))^**(hilbertFunction(d,M))-
(set{0})^**(hilbertFunction(d,M)),toList),x->basis(d,M)*vector x),
z->ideal(flatten entries z)),x-> if not 
quotient(I,x)==I then degree ideal(I,x) else 0)--The function h(d) 
--gives the minimum distance in degree d
init=ideal(leadTerm gens gb I)
hilbertFunction(1,M),hilbertFunction(2,M),hilbertFunction(3,M)
f=(x)-> if not quotient(init,x)==init then degree ideal(init,x) else 0
fp=(d) ->degree M -max apply(flatten entries basis(d,M),f)--The
--function fp(d) gives the footprint in degree d
h(1), h(2), fp(1), fp(2)
\end{verbatim}
\end{example}

\begin{example}\label{footprint=0-unmixed-ti-zero-div}\rm 
Let $S=\mathbb{F}_3[t_1,t_2,t_3,t_4]$ be a polynomial ring over the field
$\mathbb{F}_3$ with the GRevLex order $\prec$, let $\mathfrak{p}_1,\ldots,\mathfrak{p}_5$ be the 
prime ideals 
$$
\begin{array}{lll}
\mathfrak{p}_1=(t_3 + t_4, t_2 + t_4, t_1 + t_4),& \mathfrak{p}_2=(t_3
+ t_4, t_2, t_1 - t_4),& 
\mathfrak{p}_3=(t_4, t_2, t_1),\\
 \mathfrak{p}_4=(t_4, t_3, t_1),&
\mathfrak{p}_5=(t_4, t_2 - t_3, t_1),& 
\end{array}
$$
and let $I=\cap_{i=1}^5\mathfrak{p}_i$ be the intersection of these
prime ideals. Then, using {\it Macaulay\/}$2$ \cite{mac2}, we get 
${\rm reg}(S/I)=2$, $\deg(S/I)=5$, the initial ideal of $I$
is
$$
{\rm in}_\prec(I)=(t_3t_4,\, t_1t_4,\, t_1t_3,\, t_1t_2,\,t_1^2,\,t_2^2t_4,\,
t_2^2t_3),
$$
${\rm in}_\prec(I)$ is a monomial ideal of height $3$, $\mathfrak{m}$ is an
associated prime of ${\rm in}_\prec(I)$, and ${\rm
fp}_I(1)=0$. Thus the lower bound for the footprint 
${\rm fp}_I(d)$ given in Proposition~\ref{md-unmixed-propo}(b) is sharp.
\end{example}

\bibliographystyle{plain}

\end{document}